\documentclass{amsart} 
\usepackage{amssymb,latexsym,textcomp}

\usepackage[OT2,T1]{fontenc}

\theoremstyle{plain}
\newtheorem{theorem}{Theorem}[section]
\newtheorem{theoremX}{Theorem}

\newtheorem{corollary}[theorem]{Corollary}

\theoremstyle{definition}
\newtheorem{definition}{Definition}[section]

\theoremstyle{remark}
\newtheorem{remark}{Remark}[section]

\DeclareMathOperator*{\supess}{sup\,ess}
\newcommand{\N}{\mathbf N}

\newcommand{\R}{\mathbf R}

\newcommand{\e}{\varepsilon}
\renewcommand{\d}{\mathrm d}
\newcommand{\skp}[2]{\left<#1,#2\right>}

\newcommand{\luin}{\left|\!\left|\!\left|}
\newcommand{\ruin}{\right|\!\right|\!\right|}

\numberwithin{equation}{section} 

\begin{document}
\title{Operator versions of H\"older inequality and Hilbert $C^*$-modules} 

\author{Dragoljub J. Ke\v cki\' c}
\address{University of Belgrade\\ Faculty of Mathematics\\ Student\/ski trg 16-18\\ 11000 Beograd\\ Serbia}

\email{keckic@matf.bg.ac.rs}

\thanks{The author was supported in part by the Ministry of education and science, Republic of Serbia, Grant \#174034.}

\begin{abstract} Recently proved weighted Cauchy Scwarz inequality for Hilbert $C^*$-modules leads to many H\"older type inequalities for unitarily invariant norms on Hilbert space operators.
\end{abstract}


\subjclass[2010]{Primary: 47A30, Secondary: 47B10, 43A25}

\keywords{H\"older inequality, unitarily invariant
norm, elementary operator, inner type transformations, Hilbert modules.}

\maketitle

\section{Introduction}

Operator inequalities have been studied for many years. Discrete version of Cauchy-Schwarz inequality can be tracked back to \cite{SmithJFA1991} in connections with norm estimate of elementary operators and later in \cite{BhatiaDavisLAA1995}, \cite{BourinLAA2006} and many others. See, also recent work \cite{ABFM-AFA2015} and references therein. Continuous version in connection with Birman-Solomyak double operator integrals was considered in \cite{BananaJFA2005}.

In the last ten years, there was obtained many inequalities in the framework of Hilbert $C^*$-modules. Among others there are \cite{ABS-JMAA2011} or \cite{FFMS-JMAA2012}. In \cite{KeckicCS} we established that operator Cauchy-Schwarz inequalities can be derived from the corresponding Hilbert modules inequalities in a much more easier way then previously.

H\"older type inequalities (their discrete versions) was considered in \cite{AlbadawiPositivity2012} and \cite{AlbadawiBJMA2013} (see also references therein).

The aim of this note is to use weighted Cauchy-Scwarz inequality in Hilbert modules \cite[Theorem 5.1.]{SeoAFA2014} to obtain a very short and easy proof of many H\"older type inequalities for Hilbert space operators. They comprises known results (predominantly in discrete case) as well as completely new inequalities.

\section{Preliminaries}

To prove our main result we need the following:

\begin{definition} Let $A$ be a $C^*$-algebra and let $a$, $b\in A$ be positive elements. Its geometric mean $a\sharp b$ is defined as
$$a\sharp b=a^{1/2}(a^{-1/2}ba^{-1/2})^{1/2}a^{1/2},$$
if $a$ is invertible and $a\sharp b=\lim_{\varepsilon\downarrow 0+}(a+\varepsilon)\sharp b$ otherwise.
\end{definition}

It is well known that geometric mean has the following properties:

\begin{enumerate}
\item $a\sharp b=(ab)^{1/2}$ if $a$ and $b$ commute;
\item $a\sharp b=b\sharp a$;
\item\label{Positive} $a\sharp b=\max\left\{0\le x\in A\;|\;\left[\begin{matrix}a&x\\x&b\end{matrix}\right]\ge0\right\}$.
\end{enumerate}

For more details see \cite{Ando1978}. For the third property, see also \cite{PuszWoronowicz}.

Our main result relies on the following two theorems, first established by and the second by Horn and Mathias.

{\renewcommand{\thetheoremX}{S}%
\begin{theoremX}\label{SeoThm} \emph{(\cite[Theorem 5.1]{SeoAFA2014})} Let $T$ be a bounded adjointable operator on a Hilbert $C^*$-module $M$ that has a polar decomposition, let $x\in M$ be arbitrary, let $f$, $g:[0,+\infty)\to[0,+\infty)$ be continuous functions such that $f(t)g(t)\equiv t$. If $\skp x{Ty}=u|\skp x{Ty}|$ then
\begin{equation}\label{SeoWeighted}
|\skp x{Ty}|\le u^*\skp x{f(|T^*|)^2x}u\sharp\skp y{g(|T|)^2y}.
\end{equation}
\end{theoremX}%
}

\begin{proof} This is proved in \cite{SeoAFA2014} for $f(t)=t^\alpha$, $g(t)=t^\beta$, $\alpha$, $\beta\ge0$, $\alpha+\beta=1$, and for $T$ that has adjoint and such that the closures of the ranges of both $T$ and $T^*$ are complemented. However, the proof relies only on the fact that $T$ has a polar decomposition $T=U|T|$. 

For the convenience of the reader, we shall give the outline of the proof. First, recall the inequality
\begin{equation}\label{CSsharp}
\left|\skp xy\right|\le u^*\skp xxu\sharp\skp yy,\qquad\skp xy=u\left|\skp xy\right|,
\end{equation}
proved in \cite{FFMS-JMAA2012}. Next,
$$\left|\skp x{Ty}\right|=\left|\skp{x}{U|T|y}\right|=\left|\skp{f(|T|)U^*x}{g(|T|)y}\right|,$$
and apply (\ref{CSsharp}):
\begin{multline*}\left|\skp x{Ty}\right|\le u^*\skp{f(|T|)U^*x}{f(|T|)U^*x}u\sharp\skp{g(|T|)y}{g(|T|)y}=\\
    =u^*\skp{x}{Uf(|T|)^2U^*x}u\sharp\skp{y}{g(|T|)^2y}.
\end{multline*}

Finally, note that $Up(|T|)U^*=p(|T^*|)$ for all polynomials that contain only even terms, and henceforth for all continuous functions by standard limit process.
\end{proof}

\begin{remark}It will be important that $u$ can be extended either to isometry or to coisometry i.e.\ either $u^*u=1$ or $uu^*=1$.
\end{remark}

{\renewcommand{\thetheoremX}{HM}%
\begin{theoremX}\label{HornThm} \emph{(\cite[Theorem 2.3 and remark after if -- formula (2.11)]{HornLAA1990})}
Let $A$, $B$ and $X$ be operators on a Hilbert space $H$ such that $\left[\begin{matrix}A&X\\X^*&B\end{matrix}\right]\ge0$. Then for all unitarily invariant norms $\luin\cdot\ruin$, and all $p$, $q$, $r>0$ such that $1/p+1/q=1$ there holds
\begin{equation}\label{HornHolder}
\luin\,|X|^r\,\ruin\le\luin A^{pr/2}\ruin^{1/p}\luin B^{qr/2}\ruin^{1/q}.
\end{equation}
\end{theoremX}%
}

\begin{remark} This theorem is proved for matrices. However, carefully reading the proof we can see that everything works in the framework of Hilbert space operators, as well.
\end{remark}

\section{Unitarily invariant norm inequalities}

The following general theorem leads to a large number of H\"older type inequalities.

\begin{theorem}\label{glavna} Let $M$ be a Hilbert $C^*$-module over the algebra $B(H)$ of all bounded operators on some Hilbert space, let $x$, $y\in M$ and let $T:M\to M$ be adjointable operator that has a polar decomposition. If $f$, $g:[0,+\infty)\to[0,+\infty)$ are continuous functions such that $f(t)g(t)\equiv t$ then
\begin{equation}\label{HilbertPQR}
\luin\,\left|\skp{x}{Ty}\right|^r\,\ruin\le\luin\skp{x}{f(|T^*|)^2x}^{pr/2}\ruin^{1/p}\luin\skp{y}{g(|T|)^2y}^{qr/2}\ruin^{1/q},
\end{equation}
for all unitarily invariant norms $\luin\cdot\ruin$ and all $p$, $q$, $r>0$ such that $1/p+1/q=1$.
\end{theorem}

\begin{proof} Suppose $\skp x{Ty}=u\left|\skp x{Ty}\right|$ where $u$ is extended to an coisometry, i.e.\ $uu^*=1$. Denote
$$X=\skp{x}{f(|T^*|)^2x},\qquad Y=\skp{y}{g(|T|)^2y}.$$
By Theorem \ref{SeoThm} we have $\left|\skp x{Ty}\right|\le u^*Xu\sharp Y$. Next, by property (\ref{Positive}) of geometric mean, we have
$$\left[\begin{matrix}u^*Xu&\left|\skp x{Ty}\right|\\\left|\skp x{Ty}\right|&Y\end{matrix}\right]\ge0.$$
and therefore
$$0\le\left[\begin{matrix}u&0\\0&1\end{matrix}\right]\left[\begin{matrix}u^*Xu&\left|\skp x{Ty}\right|\\\left|\skp x{Ty}\right|&Y\end{matrix}\right]\left[\begin{matrix}u^*&0\\0&1\end{matrix}\right]=\left[\begin{matrix}X&\skp x{Ty}\\\skp x{Ty}^*&Y\end{matrix}\right]$$
Hence by Theorem \ref{HornThm}, formula \ref{HilbertPQR} follows.

If $u$ can not be extended to a coisometry then we pick $\skp x{Ty}^*=\skp{y}{T^*x}$ which has a polar decomposition $\skp x{Ty}^*=v\left|\skp x{Ty}^*\right|$ with $v$ extendable to a coisometry. Now result follows by interchanging roles of $x$ and $y$, $T$ and $T^*$, $f$ and $g$.
\end{proof}

Next, we give corollaries.

\begin{corollary}\label{PosledicaNeprekidna} Let $(\Omega,\mu)$ be a measurable space. We consider the Hilbert module $L^2(\Omega,\mu)$ consisting of all weakly measurable families $A_t\in B(H)$, $t\in\Omega$ such that $\int_\Omega A_t^*A_t\d\mu(t)$ exists in the weak-$*$ sense, see \cite[Example 2.3.]{KeckicCS}.

\emph{($i$)} If $A_t$, $B_t\in L^2(\Omega)$, $X_t\in L^\infty(\Omega)$, i.e.\ $\sup\|X_t\|<+\infty$ then
$$\luin\,\left|\int_\Omega A_t^*X_tB_t\d\mu(t)\right|^r\,\ruin\le\luin\left(\int_\Omega A_t^*f(|X_t^*|)^2A_t\right)^{\frac{pr}2}\ruin^{\frac1p}\luin\left(\int_\Omega B_t^*g(|X_t|)^2B_t\d\mu(t)\right)^{\frac{qr}2}\ruin^{\frac1q}.$$
In particular
$$\luin\,\left|\int_\Omega A_t^*X_tB_t\d\mu(t)\right|^r\,\ruin\le\luin\left(\int_\Omega A_t^*|X_t^*|^{2\alpha}A_t\right)^{\frac{pr}2}\ruin^{\frac1p}\luin\left(\int_\Omega B_t^*|X_t|^{2\beta}B_t\d\mu(t)\right)^{\frac{qr}2}\ruin^{\frac1q},$$
where $\alpha$, $\beta\ge0$, $\alpha+\beta=1$.

\emph{($ii$)} If, in addition, $\mu(\Omega)<+\infty$ and $r\ge2$ then
$$\luin\left|\int_\Omega A_t^*X_tB_t\d\mu(t)\right|^r\ruin\le\luin\int_\Omega |A_t|^{rp}\d\mu(t)\ruin^{\frac1p}\luin\int_\Omega |B_t|^{rq}\d\mu(t)\ruin^{\frac1q}\supess_{t\in\Omega}\|X_t\|^r.$$
The latter can be extended to all $A_t$, $B_t$ such that $\int_\Omega|A_t|^p\d\mu(t)$, $\int_\Omega|B_t|^q\d\mu(t)$ exist as weak-$8$ integrals. In particular, for any $Q$ norm $\|\cdot\|_Q$ we have
$$\left\|\int_\Omega A_t^*X_tB_t\d\mu(t)\right\|_Q\le\left\|\int_\Omega |A_t|^{p}\d\mu(t)\right\|_Q^{\frac1p}\left\|\int_\Omega |B_t|^{q}\d\mu(t)\right\|_Q^{\frac1q}\supess_{t\in\Omega}\|X_t\|.$$

\emph{($iii$)} If $A_t$ and $B_t$ are families consisting of mutually commuting normal operators, $X\in B(H)$ arbitrary and $\mu(\Omega)<+\infty$, then for all $Q$-norms $\|\cdot\|_Q$ there holds
$$\left\|\int_\Omega A_t^*XB_t\d\mu(t)\right\|_Q\le\left\|\left(\int_\Omega |A_t|^{p}\d\mu(t)\right)^{\frac1p}X\left(\int_\Omega |B_t|^{q}\d\mu(t)\right)^{\frac1q}\right\|_Q.$$
\end{corollary}

\begin{proof} We have:

($i$) Consider the mapping $T_X:L^2(\Omega,\mu)\to L^2(\Omega,\mu)$ given by
$$T_X(A_t)=(X_tA_t).$$
It is easy to see that $T_X$ is bounded, adjointable, $\|T_x\|=\supess\|X_t\|$, as well as that $T_X$ has a polar decomposition $T_X=T_UT_{|X|}$, where $X_t=U_t|X_t|$ is the polar decomposition of $X_t$. Apply Theorem \ref{glavna}.

($ii$) First, note that $A_t^*|X^*_t|^{2\alpha}A_t\le A_t^*A_t\|X^*_t\|^{2\alpha}\le|A_t|^2\supess\|X_t\|^{2\alpha}$ and similar for $B_t$. Then recall the Jensen inequality
$$\luin\varphi\left(\int_\Omega A_t\d\mu(t)\right)\ruin\le\luin\int_\Omega\varphi(A_t)\d\mu(t)\ruin,$$
which holds for convex functions $\varphi:[0,+\infty)\to[0,+\infty)$, pointwise positive $A_t$ and $\mu(\Omega)=1$, see \cite[Proposition 2.5.]{KeckicBJMA}, and note that $t\mapsto t^{pr/2}$, $t^{qr/2}$ are convex for $r\ge 2$, $p$, $q\ge1$. Finally, note that both sides of the required inequality are homogeneous and hence $\mu(\Omega)$ can be canceled.

For the second inequality, recall that unitarily invariant norm $\luin\cdot\ruin$ is called $Q$-norm, if there is an other unitarily invariant norm $\|\cdot\|$ such that $\luin A\ruin=\||A|^2\|^{1/2}$. Now, take $r=2$ in the previous inequality and take a power to the $1/2$.

($iii$) For $X_t\equiv X$ and $\int_\Omega|A_t|^p\d\mu(t)$, $\int_\Omega|B_t|^q\d\mu(t)\le I$ the result follows from the previous item. In general case, note that $A_t$ as well as $A_t^*$ commutes with $\int_\Omega|A_t|^p\d\mu(t)$. Choose $\e>0$ and apply the special case to
$$A_t\left(\int_\Omega|A_t|^p\d\mu(t)+\e I\right)^{-1/p},\qquad B_t\left(\int_\Omega|B_t|^q\d\mu(t)+\e I\right)^{-1/q}\qquad\mbox{and}$$
$$\left(\int_\Omega|A_t|^p\d\mu(t)\right)^{1/p}X\left(\int_\Omega|B_t|^q\d\mu(t)\right)^{1/q}$$
instead of $A_t$, $B_t$ and $X$. We obtain
\begin{multline*}\left\|\int_\Omega A_t^*XB_t\d\mu(t)\right\|_Q=\\=\left\|\int_\Omega A_t^*\left(\int_\Omega+\e I\right)^{-\frac1p}\left(\int_\Omega+\e I\right)^{\frac1p}X\left(\int_\Omega+\e I\right)^{\frac1q}\left(\int_\Omega+\e I\right)^{-\frac1q}B_t\d\mu(t)\right\|_Q\le\\\le\left\|\left(\int_\Omega |A_t|^{p}\d\mu(t)+\e I\right)^{\frac1p}X\left(\int_\Omega |B_t|^{q}\d\mu(t)+\e I\right)^{\frac1q}\right\|_Q.
\end{multline*}
Finally, let $\e\to0$. (Notice the well known lower semicontinuity of unitarily invariant norms, with respect to weak limit.)
\end{proof}

\begin{corollary} Let us consider the discrete case. Let $\gamma_n\in\R^+$ be a sequence, $A_n$, $B_n$, $X_n\in B(H)$, and let $p$, $q$, $r>0$, $1/p+1/q=1$.

\emph{($i$)} \cite[Theorem 21.]{AlbadawiPositivity2012} (see also Theorem 15 of the same reference) If $\sum_{n=1}^{+\infty}\gamma_n|A_n|^2$, $\sum_{n=1}^{+\infty}\gamma_n|B_n|^2$ converge weakly, $\sup\|X_n\|<+\infty$ and $f$, $g:[0,+\infty)\to[0,+\infty)$ are continuous functions such that $f(t)g(t)\equiv t$ then
$$\luin\,\left|\sum_{n=1}^{+\infty}\gamma_nA_n^*X_nB_n\right|^r\,\ruin\le\luin\left(\sum_{n=1}^{+\infty} \gamma_nA_n^*f(|X_n^*|)^2A_n\right)^{\frac{pr}2}\ruin^{\frac1p}\luin\left(\sum_{n=1}^{+\infty} \gamma_nB_n^*g(|X_n|)^2B_n\right)^{\frac{qr}2}\ruin^{\frac1q}.$$

In particular, for $\alpha$, $\beta\ge0$, $\alpha+\beta=1$, we have
\begin{equation}\label{peta}\luin\,\left|\sum_{n=1}^{+\infty}\gamma_nA_n^*X_nB_n\right|^r\,\ruin\le\luin\left(\sum_{n=1}^{+\infty} \gamma_nA_n^*|X_n^*|^{2\alpha}A_n\right)^{\frac{pr}2}\ruin^{\frac1p}\luin\left(\sum_{n=1}^{+\infty} \gamma_nB_n^*|X_n|^{2\beta}B_n\right)^{\frac{qr}2}\ruin^{\frac1q}.
\end{equation}

\emph{($ii$)} If, in addition, $\sum_{n=1}^{+\infty}\gamma_n=1$ and $r\ge2$ then
$$\luin\sum_{n=1}^{+\infty}\left|\gamma_nA_n^*X_nB_n\right|^r\ruin\le\luin\sum_{n=1}^{+\infty} \gamma_n|A_n|^{pr}\ruin^{\frac1p}\luin\sum_{n=1}^{+\infty}\gamma_n|B_n|^{qr}\ruin^{\frac1q}\sup_{n\ge1}\|X_n\|.$$
The latter can be extended to all $A_n$, $B_n$ such that $\sum|A_n|^p$, $\sum |B_n|^q$ converge weakly. In particular, for any $Q$ norm $\|\cdot\|_Q$ we have
$$\left\|\sum_{n=1}^{+\infty}\gamma_nA_n^*X_nB_n\right\|_Q\le\left\|\sum_{n=1}^{+\infty} \gamma_n|A_n|^p\right\|_Q^{\frac1p}\left\|\sum_{n=1}^{+\infty}\gamma_n|B_n|^q\right\|_Q^{\frac1q}\sup_{n\ge1}\|X_n\|.$$

\emph{($iii$)} For all unitarily invariant norms $\luin\cdot\ruin$ and for $p\ge2$ there holds
$$\luin\sum_{n=1}^{+\infty}\gamma_nA_n^*X_nB_n\ruin\le\luin\sum_{n=1}^{+\infty} \gamma_n|A_n|^p\ruin^{\frac1p}\luin\sum_{n=1}^{+\infty}\gamma_n^{q/2}|B_n|^q\ruin^{\frac1q}\sup_{n\ge1}\|X_n\|.$$
In particular, for $m\in\N$ we have:
$$\luin\sum_{n=1}^mA_n^*X_nB_n\ruin\le m^{\left|\frac12-\frac1p\right|}\luin\sum_{n=1}^m |A_n|^p\ruin^{\frac1p}\luin\sum_{n=1}^{+\infty}|B_n|^q\ruin^{\frac1q}\sup_{n\ge1}\|X_n\|.$$
The last inequality reduces to \cite[Theorem 17]{AlbadawiPositivity2012} for $X_n\equiv I$. Note that the constant $m^{\left|\frac12-\frac1p\right|}$ is sharp, see \cite[Remark 18]{AlbadawiPositivity2012}.

\emph{($iv$)} If $A_n$ and $B_n$ are sequences of mutually commuting normal operators, $X\in B(H)$ arbitrary and $\sum_{n=1}^{+\infty}\gamma_n<+\infty$ then for all $Q$-norms $\|\cdot\|_Q$ there holds
$$\left\|\sum_{n=1}^{+\infty}\gamma_nA_n^*XB_n\right\|_Q\le\left\|\left(\sum_{n=1}^{+\infty}\gamma_n |A_n|^{p}\right)^{\frac1p}X\left(\sum_{n=1}^{+\infty}\gamma_n|B_n|^{q}\right)^{\frac1q}\right\|_Q.$$

\emph{($v$)} If $A_n$ and $B_n$ are sequences of mutually commuting normal operators, $X\in B(H)$ arbitrary and $m\in\N$ then for all unitarily invariant norms $\luin\cdot\ruin$ there holds
$$\luin\sum_{n=1}^mA_n^*XB_n\ruin\le m^{\left|\frac12-\frac1p\right|}\luin\left(\sum_{n=1}^m |A_n|^{p}\right)^{\frac1p}X\left(\sum_{n=1}^m|B_n|^{q}\right)^{\frac1q}\ruin.$$

\end{corollary}

\begin{proof} 

($i$), ($ii$) and ($iv$) are special cases of ($i$), ($ii$) and ($iii$), respectively of the previous corollary, for $\Omega=\N$, $\mu(E)=\sum_{n\in E}\gamma_n$.

($iii$) Put $r=1$ and $X_n=I$ in (\ref{peta}). The function $t\mapsto t^{p/2}$ is convex and $t\mapsto t^{q/2}$ is concave. Apply, aforementioned Jensen inequality to the first factor, and the inequality
$$\luin\varphi(\sum B_n)\ruin\le\luin\sum\varphi(B_n)\ruin$$
which holds for concave $\varphi$ (see \cite[Lemma 2.4]{KeckicBJMA}) to the second factor.

($v$) This can be derived from ($iii$) in the same way as in the previous corollary the third assertion is derived form the first.

\end{proof}

\section{An inequality in operator ordering}

In the paper \cite{SeoAFA2014} an other inequality in Hilbert modules was proved, namely

{\renewcommand{\thetheoremX}{S2}%
\begin{theoremX}\emph{(\cite[Theorem 4.1]{SeoAFA2014})}
Let $p>1$, $1/p+1/q=1$ and let $A$, $B$ be positive adjointable operators on some Hilbert $C^*$-module over a $C^*$-algebra. Then
$$\skp x{B^q\sharp_{1/p}A^px}\le\skp x{B^qx}\sharp_{1/p}\skp x{A^px},$$
where $\sharp_\theta$ stands for weighted geometric mean:
$$A\sharp_\theta B:=A^{1/2}(A^{-1/2}BA^{-1/2})^\theta A^{1/2},$$
for $A$ invertible, and $A\sharp_\theta B=\lim\limits_{\varepsilon\downarrow0}(A+\varepsilon I)\sharp_\theta B$, otherwise.
\end{theoremX}
}

Applying this to "diagonal" operators $X_t\mapsto A_tX_t$ and $X_t\mapsto B_tX_t$ for pointwise positive and weakly measurable families $A_t$ and $B_t$, we obtain the following corollary:
\begin{corollary} We have

\begin{enumerate}

\item For a measurable space $(\Omega,\mu)$ and weakly measurable families $A_t$, $B_t$, of positive Hilbert space operators, $p$, $q\ge1$, $1/p+1/q=1$ we have
$$\int_\Omega B_t^q\;\sharp_{1/p}\;A_t^p\d\mu(t)\le\left(\int_\Omega B_t^q\d\mu(t)\right)\;\sharp_{1/p}\;\left(\int_\Omega A_t^p\d\mu(t)\right)$$

\item For sequences $A_n$, $B_n$, of positive Hilbert space operators, $p$, $q\ge1$, $1/p+1/q=1$ we have
$$\sum_{n=1}^{+\infty} B_n^q\;\sharp_{1/p}\;A_n^p\le\left(\sum_{n=1}^{+\infty}B_n^q\right)\;\sharp_{1/p}\;\left(\sum_{n=1}^{+\infty}A_n^p\right)$$
\end{enumerate}
\end{corollary}
\bibliographystyle{siam}
\bibliography{Holder}

\end{document}